\documentclass[11pt]{article} 
\usepackage{cancel}
\usepackage[english]{babel}
\usepackage[utf8x]{inputenc}
\usepackage{amsmath}
\usepackage{amsthm}
\usepackage{graphicx}
\usepackage{color}
\usepackage{amsfonts}
\usepackage{epstopdf}
\usepackage{float}
\usepackage{hyperref}
\usepackage[dvipsnames]{xcolor}
\usepackage{caption}
\usepackage{subcaption}
\usepackage{tikz}
\usepackage{makeidx}
\usepackage{amssymb}
\usepackage{verbatim} 

\usepackage{hyperref}  

\newtheorem*{hypothesis}{Hypothesis (H)}

\newcommand{\HypH}{\hyperref[hyp:H]{(H)} }

\makeindex

\let \dis \displaystyle

\let \bb \mathbb

\let \cal \mathcal

\let \vp \varphi

\let\frak\mathfrak

\newcommand{\R}{\mathbb{R}} 

\usepackage{dsfont}
\usepackage{algorithm}%
\usepackage{algorithmicx}%
\usepackage{algpseudocode}
\usepackage{soul}

\title{ Averaged Controllability of the Random Schrödinger Equation with Diffusivity Following Absolutely Continuous Distributions}
\author{Jon Asier Bárcena-Petisco\thanks{Department of Mathematics, University of the Basque Country UPV/EHU, Barrio Sarriena s/n, 48940, Leioa, Spain, \url{https://orcid.org/0000-0002-6583-866X}. E-mail: {\tt jonasier.barcena@ehu.eus}.} \and Fouad Et-Tahri\thanks{Lab-SIV, Faculty of Sciences-Agadir, Ibnou
Zohr University, B.P. 8106, Agadir,
Morocco, \url{https://orcid.org/0009-0006-8926-5815}. Email: {\tt fouad.et-tahri@edu.uiz.ac.ma}.}}%

\begin{document}
\maketitle
\numberwithin{equation}{section}
\newtheorem{theorem}{Theorem}
\numberwithin{theorem}{section}
\newtheorem{proposition}[theorem]{Proposition}
\newtheorem{conjecture}{Conjecture}
\newtheorem{fact}[theorem]{Fact}
\newtheorem{lemma}[theorem]{Lemma}
\newtheorem{step}{Step}
\newtheorem{corollary}[theorem]{Corollary}
\theoremstyle{remark}
\newtheorem{remark}[theorem]{Remark}
\newtheorem{definition}[theorem]{Definition}
\newtheorem{example}[theorem]{Example}

\noindent
\textbf{Abstract:} 
This paper is devoted to the averaged controllability of the random Schrödinger equation, with diffusivity as a random variable drawn from a general probability distribution. First, we show that the solutions to these random Schrödinger equations are null averaged controllable with an open-loop control independent of randomness from any arbitrary subset of the domain with strictly positive measure and in any time. This is done for an interesting class of random variables, including certain stable distributions, specifically recovering the known result when the random diffusivity follows a normal or Cauchy distribution. Second, by the Riemann-Lebesgue lemma, we prove for any time the lack of averaged exact controllability in a $L^2$ setting for all absolutely continuous random variables. Notably, this implies that this property is not inherited from the exact controllability of the Schrödinger equation. Third, we show that simultaneous null controllability is not possible except for a finite number of scenarios. 
Finally, we perform numerical simulations that robustly validate the theoretical results.


\noindent
\textbf{Abbreviated title:} Averaged Controllability of Schrödinger Equation


\noindent
\textbf{Acknowledgements:} J.A.B.P was supported by the Grant PID2021-126813NB-I00 
	funded by \\ MCIN/AEI/10.13039/501100011033 and by ``ERDF A way of making 
	Europe" and by the grant\\IT1615-22 funded the Basque Government.

\noindent
\allowdisplaybreaks
\newpage

\section{Introduction}

\paragraph{}

In this paper, we study the averaged controllability properties of the Schrödinger equation. The dynamics of the system is given by:

\begin{equation}\label{sys:Schr}  (\mathcal{P}_{\alpha}) \quad
\begin{cases}
\partial_t y -\alpha \mathrm{i}\Delta y=\mathds{1}_{G_0} u,& \mbox{ on }Q_T,\\
y=0, & \mbox{ on }\Sigma_T,\\
y(0,\cdot)=y_0, & \mbox{ in } G.
\end{cases}
\end{equation}
Here $T>0$ is a time horizon, $G \subset \R ^{d}$ ($d\geq 1$) is a Lipschitz domain with boundary $\partial G$, $Q_T \overset{\text{def}}{=} (0,T)\times G$, $\Sigma_T \overset{\text{def}}{=} (0,T)\times \partial G$, $G_0\subset G$ is a subset of strictly positive measure, $\mathds{1}_{G_0}$ its indicator function, $\mathrm{i}$ stands for the imaginary unit satisfying $\mathrm{i}^2=-1$, and 
$\alpha=\alpha(\omega)$ is an unknown real random variable on a probability space $(\Omega,\mathcal{F},\mathbb{P})$
which induces a probability measure $\mu_{\alpha}$ in $\mathbb{R}$, the input (control) $u=u(t,x)\in L^{2}((0,T)\times G_0)$ and the initial state $y_0\in L^{2}(G)$ are independent of the parameter $\omega$. It is well known that the operator $\mathrm{i}\Delta$, with Dirichlet boundary conditions, generates a unitary group $(e^{\mathrm{i}t\Delta})_{t\in\mathbb{R}}$ on $L^{2}(G)$, and 
for $\mathbb{P}$-almost every $\omega\in\Omega$ and all
choices of $u$ and $y_0$, the solution $y=y(t,x;\alpha(\omega); y_0; u)$ of $(\mathcal{P}_{\alpha(\omega)})$ is given by the Duhamel formula:
\[y(t,\cdot;\alpha(\omega); y_0; u)=e^{\mathrm{i}t\alpha(\omega)\Delta}y_0+ \int_{0}^{t}e^{\mathrm{i}(t-s)\alpha(\omega)\Delta}\mathds{1}_{G_0}u(s,\cdot)ds \in\mathcal{C}([0,T];L^{2}(G)).\]
The ideal situation would be to have
a simultaneous control for $\mathbb{P}$-almost every $\omega\in\Omega$. As we shall see, this is not possible with any absolutely continuous random variable (see Theorem \ref{thm:nonsimcon}). 
However, we can expect to control at least the average (or expectation) of $L^{2}(G)$-valued Borel random variable $\omega\mapsto y(T,\cdot;\alpha(\omega); y_0; u)\in L^{\infty}(\Omega;L^{2}(G))$. The average of such system exists, due to the boundedness of the state with respect to the randomness, and will be denoted by $\mathbb{E}(y(t, \cdot;\alpha; y_0; u))$; that is,
\begin{eqnarray*}
    \mathbb{E}(y(t,\cdot;\alpha; y_0; u))& \overset{\text{def}}{=} &\int_{\Omega}y(t,\cdot;\alpha(\omega); y_0; u)d\mathbb{P}(\omega) \in\mathcal{C}([0,T];L^{2}(G)).
\end{eqnarray*}
By the expectation transformation formula (see  \eqref{expectation transformation formula} below), we obtain:
\begin{eqnarray*}
    \mathbb{E}(y(t,\cdot;\alpha; y_0; u))=\int_{-\infty}^{\infty}y(t,\cdot;\xi; y_0; u) d\mu_{\alpha}(\xi).
\end{eqnarray*}
where, for all $\xi\in\mathbb{R}$, $y(t,\cdot;\xi; y_0; u)$ is the solution of $(\mathcal{P}_{\xi})$. 
\par 
In probability theory, the distribution of a random variable $\alpha$ is characterized by its characteristic function  $\varphi_{\alpha}$, given by:
\begin{equation}\label{def:charfunc}
	\vp_\alpha(s) \overset{\text{def}}{=} \mathbb{E}(e^{\mathrm{i}s\alpha})=\int_{-\infty}^{\infty}e^{\mathrm{i}s\xi}d\mu_{\alpha}(\xi).
\end{equation}
\paragraph{}
Our first contribution (see Theorem \ref{Main result}) is to establish the null averaged controllability of \eqref{sys:Schr} for an interesting class of random variables, described by the following hypothesis:
\begin{hypothesis}\label{hyp:H}
	There are $c>0$, $r>\frac{1}{2}$, $\theta>0$, and $T_0\in (0, T]$ 
    such that for all $\lambda\geq0$ 
    and $t_1,t_2\in [0,T_0]$ satisfying $t_1<t_2$ we have that: 
	\begin{equation}
		|\vp_\alpha(\lambda t_2)|\leq e^{-c\lambda^r(t_2-t_1)^{\theta}}|\vp_\alpha(\lambda t_1)|, \label{eq:hypSch}
	\end{equation}
    where $\vp_\alpha$ is the characteristic function of $\alpha$.
\end{hypothesis}\noindent
Many examples satisfying \HypH are given in Section \ref{Sec:HYP}, including normal and Cauchy random variables. This result is obtained with the Hilbert Uniqueness Method (see \cite{lions1988controlabilite,russell1978controllability}) and notably by obtaining an observability inequality with a spectral approach. In particular, we use the approach from the seminal work \cite{miller2010observability}, an approach that has been used in many relevant papers, such as \cite{apraiz2014observability,  beauchard2018null, burq2022propagation, lissy2019internal, lu2016averaged}. In the context of average controllability, it was also used for the heat equation in \cite{barcena2021averaged}.

Moreover, a second contribution is that, as we show in Theorem \ref{tm:noexcon}, the averaged exact controllability does not hold in the $L^2$ set up, which is a surprising result because one would expect that such property would be inherited from the exact controllability properties of the Schrödinger equation.

 Finally, a third contribution, as we show in Theorem \ref{thm:nonsimcon}, is that the preimage of $0$ under $\xi\mapsto y(T,\cdot;\xi;y_0;u)$ is finite for all $y_0\in L^2(G)\setminus\{0\}$ and $u\in L^2((0,T)\times G_0)$. In particular, the event: ``the system \eqref{sys:Schr} is simultaneous null controllable" is negligible for absolutely continuous random variables, that is, $$\mathbb{P}\left[\omega\;:\; y(T,\cdot;\alpha(\omega);y_0;u)=0\right]=0.$$

\noindent \textbf{A brief overview of the related literature.}  Regarding the control of averaged properties,
the problem and its abstract formulation were first presented in the work \cite{zuazua2014averaged}, where finite-dimensional systems were considered. The pioneering study on averaged controllability in the context of partial differential equations (PDEs) was introduced in \cite{lu2016averaged}. In this study, the authors addressed the problem of controlling the average of PDEs and examined the properties of the transport, heat, and Schrödinger equations in specific scenarios. For the Schrödinger equation, they analyzed the equation when the diffusivity follows uniform, exponential, normal, Laplace, chi-square, and Cauchy probability laws. 
In \cite{zuazua2016stable} the continuous average of the heat equation was considered. Further contributions, such as \cite{loheac2017averaged} and \cite{lazar2018stability} investigated perturbations of probability density functions modeled by Dirac masses. In \cite{coulson2019average} and \cite{barcena2021averaged} the controllability of the heat equation with random diffusivity was studied. In the second paper, it was discovered that some probability density functions caused fractional dynamics for the heat equation.
Moreover, we would like to point out that there are many known results for lower-order random terms, as shown in the survey \cite{lu2022concise} and the book \cite{lu2021mathematical}, and some of them are recent results for nonlinear parabolic stochastic equations such as in \cite{hernandez2022statistical} and \cite{hernandez2023global}.

Regarding the controllability of the Schrödinger equation, there are many classical papers such as \cite{anantharaman2012dispersion,  burq2004geometric, lebeau1992controle, machtyngier1994exact,
phung2001observability}, 
involving geometric control conditions. As for simultaneous control, we would like to highlight \cite{morancey2015simultaneous}, where there is a bilinear control, and \cite{lopez2016null}, where cascade-like systems are studied. Also, for stochastic equations with random lower-order terms, we may highlight, for example, \cite{lu2013observability}. Finally, as mentioned above, for random diffusivity, it was studied in \cite{lu2016averaged} for specific probability laws. 

Our contribution  to the literature is to obtain the averaged controllability of the Schrödinger equation with random diffusivity in a more general setting and to determine how much we can generalize the known results. 

\noindent \textbf{Structure of this paper.} In Section \ref{Sec2}, we introduce several key concepts along with their characterization, and analyze the main assumption of the paper. Section \ref{Sec3} is devoted to proving the main result on the averaged null controllability of the system \eqref{sys:Schr}. In Section \ref{Sec4}, we prove the lack of exact average for general absolutely continuous random variables and of simultaneous controllability. In Section \ref{Sec5}, we present a numerical method to validate our theoretical result in the case of Cauchy and normal distributions. Section \ref{sec:prob} is dedicated to the presentation of some analogue and some open problems. Finally, Section \ref{Sec6} is dedicated to the conclusion and outlines several potential directions for future research.

\noindent \textbf{Notation:} 
Throughout this paper, we adopt some standard notations and usual abbreviations in probability:
\begin{itemize}
     \item $\mathbb{N}=\{0,1,2, \cdots\}$ stands for the set of natural numbers that includes zero.
     \item For every measurable set $A\subset \mathbb{R}^d$ ($d\ge 1$), we denote by $|A|$ its Lebesgue measure. For every $x\in\mathbb{R}^d$ and $r>0$, $B(x,r)$ denotes the open ball centered at $x$ with radius $r$.
    \item $\mathcal{H}$ stands for the space $L^2(G) \overset{\text{def}}{=} L^{2}(G;\mathbb{C})$ equipped with its usual Hermitian scalar product $\left\langle\cdot,\cdot\right\rangle$, and $\|\cdot\|$ stands for its Euclidean norm.
    \item $(\Omega,\mathcal{F},\mathbb{P})$ denotes a probability space.
    \item A $\mathcal{H}$-valued Borel random variable is a function from $\Omega$ to $\mathcal{H}$ that is $(\mathcal{F},\mathcal{B}(\mathcal{H}))$-measurable, where $\mathcal{B}(\mathcal{H})$ denotes the Borel $\sigma$-algebra on $\mathcal{H}$. A real Borel random variable is simply called a real random variable.
    \item $\alpha$ denotes a real random variable.
    \item The probability distribution of $\alpha$, defined by $\mu_{\alpha} \overset{\text{def}}{=} \mathbb{P}\circ \alpha^{-1}$, is called its distribution.
    \item \textbf{PDF} denotes the probability density function of an absolutely continuous random variable. We recall that the distribution of $\alpha$
    is absolutely continuous (with respect to the Lebesgue measure) if there is $\rho_{\alpha}\in L^1(\mathbb R)$ nonnegative, called the probability density function, such that for all Borel measurable set $A\subset\mathbb R$, $$\mathbb \mu_{\alpha}[A] \overset{\text{def}}{=} \mathbb{P}[\alpha\in A]=\int_A \rho_{\alpha}(x)dx.$$
    We simply say that $\alpha$ is absolutely continuous if its distribution is absolutely continuous.
    \item Recall that, for any $F:\mathbb{R}\mapsto \mathcal{H}$ continuous function, $F\circ\alpha$ is $\mathcal{H}$-valued Borel random variable and the expectation of $F\circ\alpha$ (provided that it exists) is given by the expectation transformation formula:
    \begin{eqnarray}
        \mathbb{E}(F\circ\alpha)\overset{\text{def}}{=}\int_{\Omega}F(\alpha(\omega))d\mathbb{P}(\omega) =\int_{-\infty}^{\infty}F(\xi)d\mu_{\alpha}(\xi). \label{expectation transformation formula}
    \end{eqnarray}
    \item \textbf{CF} denotes the characteristic function of a random variable. It is denoted by $\varphi_{\alpha}$ and is defined by:
    \[\varphi_{\alpha}(s)\overset{\text{def}}{=}\mathbb{E}(e^{\mathrm{i}s\alpha})=\int_{-\infty}^{\infty}e^{\mathrm{i}s\xi}d\mu_{\alpha}(\xi).\]
    \item We denote by $\{\lambda_n\}_{n\in\mathbb{N}}$ the eigenvalues of $-\Delta$ under Dirichlet conditions and $\{e_n\}_{n\in\mathbb{N}}$ the corresponding sequence of orthonormal eigenfunctions in $L^2(G)$. Note that $\lambda_0>0$, and the sequence $\{\lambda_n\}_{n\in\mathbb{N}}$ is increasing and tends to infinity.
\end{itemize}

\section{Preliminary concepts and main hypothesis} \label{Sec2}

\subsection{Preliminary concepts}
We now introduce the following notions of averaged controllability.
\begin{definition}\label{def-ex-con}
System \eqref{sys:Schr} is said to fulfill the
property of {\it exact averaged controllability}
or to be  {\it exactly controllable in average} in 
the space $\mathcal{H}$ with control
cost $C_{ex}=C_{ex}(G, G_0, \alpha, T)>0$ if given any  $y_0, y_1\in \mathcal{H}$, there
exists a control $u\in L^2((0,T)\times G_{0})$ such that:
\begin{equation}\label{def-ex-con-eq1}
\|u\|_{L^2((0,T)\times G_0)}\leq C_{ex}(\|y_0\| + \|y_1\|)
\end{equation}
and the average of solutions to
\eqref{sys:Schr} satisfies:
\begin{equation*}
\mathbb{E}(y(T,\cdot;\alpha; y_0; u))=y_1.
\end{equation*}
\end{definition}

\begin{definition} 
System \eqref{sys:Schr} is said to fulfill the
property of {\it null averaged controllability}
or to be  {\it null controllable in average} in 
the space $\mathcal{H}$ with control
cost $C_{null}=C_{null}(G,G_0, \alpha, T)>0$ if given any  $y_0\in \mathcal{H}$, there
exists a control $u\in L^2((0,T)\times G_0)$ such that
\begin{equation}\label{est:connull}
\|u\|_{L^2((0,T)\times G_0)}\leq C_{null}\|y_0\|
\end{equation}
and the average of solutions to
\eqref{sys:Schr} satisfies:
\begin{equation}
\mathbb{E}(y(T,\cdot;\alpha; y_0; u))=0. \label{nullaver}
\end{equation}
\end{definition}

As usual, these notions admit a dual notion regarding observability. For that, note that the adjoint system is given by:

\begin{equation}\label{sys:adjSchr}
\begin{cases}
-\partial_t z +\alpha \mathrm{i} \Delta z=0,& \mbox{ on }Q_T,\\
z=0, & \mbox{ on }\Sigma_T,\\
z(T,\cdot)=z_T, & \mbox{ in } G.
\end{cases}
\end{equation}

\begin{definition} 
	System \eqref{sys:adjSchr} is said to fulfill the
	property of {\it exact averaged observability}
	or to be  {\it exactly observable in average} in 
	the space $\mathcal{H}$ with observability
	cost $C_{exob}=C_{exob}(G,G_0, \alpha, T)>0$ if for any  $z_T\in \mathcal{H}$:
	\begin{equation}\label{est:exact_obs}
		\left\|z_T\right\|\leq C_{exob}\left\|\mathbb{E}( z(\cdot,\cdot;\alpha;z_T))\right\|_{L^2((0,T)\times G_0)}.
	\end{equation}
\end{definition}

\begin{definition} 
System \eqref{sys:adjSchr} is said to fulfill the
property of {\it null averaged observability}
or to be  {\it null observable in average} in 
the space $\mathcal{H}$ with observability
cost $C_{ob}=C_{ob}(G,G_0, \alpha, T)>0$ if for any  $z_T\in \mathcal{H}$:
\begin{equation}\label{est:obs}
\left\|\mathbb{E}(z(0,\cdot;\alpha;z_T))\right\|\leq C_{ob}\left\|\mathbb{E}( z(\cdot,\cdot;\alpha;z_T))\right\|_{L^2((0,T)\times G_0)}.
\end{equation}
\end{definition}

\begin{proposition} \label{Pr: obser_null cont}
Let $T>0$.
System \eqref{sys:Schr} is null controllable in average if and only if the system \eqref{sys:adjSchr} is null observable in average. Furthermore, the optimal controllability constant and the optimal observability constant are related by:
 \begin{eqnarray}
 	C_{null}=C_{ob}. \label{optimal}
 \end{eqnarray}
\end{proposition}

\begin{remark}
The result \eqref{optimal} 
is a side result obtained within the proof, which, even if we do not use in the paper, may have some applications in Optimal Control.  
\end{remark}

Proposition \ref{Pr: obser_null cont} was stated in an abstract setting in \cite[Theorem A.2]{lu2016averaged} without proof. However, it might be proved with some adaptation of the Hilbert Uniqueness Method. This method dates back to  \cite{lions1988controlabilite,russell1978controllability}, 
and, in particular, for the averaged controllability problem, it dates back to \cite[Theorem 3]{zuazua2014averaged}
and \cite[Theorem A.1]{lu2016averaged}. Since some modifications  from the literature are needed (see Remark \ref{rk:diffproof}),  we sketch the proof pointing out the main novelties:

\begin{proof}[Sketch of the proof]
Performing classical integration by parts, it can be easily proved that, for a fixed $u\in L^2((0,T)\times G_0)$: 
	\begin{equation}
			\left\langle \mathbb{E}(y(T,\cdot;\alpha; y_0; u)),z_T\right\rangle-\left\langle y_0, \mathbb{E}(z(0,\cdot;\alpha;z_T))\right\rangle
			=\int_0^T\int_{G_0} u(t,x)\overline{\mathbb{E}(z(t,x;\alpha;z_T))}dxdt. \label{eq: carcontrol}
	\end{equation}
    \paragraph{}
Let us suppose that \eqref{sys:Schr} is null controllable in average. It is standard to show that $C_{ob}\leq C_{null}$. 
\paragraph{}
	Conversely, let us suppose that the observability estimate \eqref{est:obs} is satisfied.
    We consider the linear subspace $\mathcal{F}$ of $L^{2}((0,T)\times G_0)$:
	\begin{eqnarray*}
		\mathcal{F}\overset{\text{def}}{=}\{ \mathds{1}_{G_0}\overline{\mathbb{E}(z(\cdot,\cdot;\alpha;z_{T}))},\quad z_{T}\in \mathcal{H}\},
	\end{eqnarray*}
	and the linear functional $\Phi$ whose domain is $\mathcal{F}$ and defined by:
	\begin{eqnarray*}
		\Phi(\mathds{1}_{G_0}\overline{\mathbb{E}(z(\cdot,\cdot;\alpha;z_{T}))})\overset{\text{def}}{=}-\left\langle y_{0}, \mathbb{E}(z(0,\cdot;\alpha;z_{T})) \right\rangle.
	\end{eqnarray*}
	Using the observability inequality \eqref{est:obs}, we obtain that $\Phi$ is well defined and a bounded linear functional on $\mathcal{F}$ with norm: $$\|\Phi\|_{\mathcal{F}^{\prime}}\leq C_{ob}\|y_{0}\|.$$
	Then, by extending by density to the closure of $\mathcal{F}$ and setting 
	$0$ on the orthogonal set (in $L^{2}((0,T)\times G))$) of the closure of $\mathcal{F}$, we can extend $\Phi$ to a bounded linear functional $\widetilde{\Phi}$ on $L^{2}((0,T)\times G_0)$ with the same norm. It follows from the Riesz representation Theorem that there is a $u\in L^{2}((0,T)\times G_0)$ such that for any $v\in L^{2}((0,T)\times G_0)$,
	\begin{eqnarray*}
		\widetilde{\Phi}(v)=\left\langle u,v \right\rangle_{L^{2}((0,T)\times G_0)}.
	\end{eqnarray*}
	In particular, it is standard that $u$ takes the average of the solutions takes $0$ at time $T$ and we have that:
\begin{equation*}
	\|u\|_{L^2((0,T)\times G_0)}=\|\widetilde{\Phi}\|_{(L^2((0,T)\times G_0))^{\prime}}=\|\Phi\|_{\mathcal{F}^{\prime}}\leq C_{ob}\|y_0\|,
\end{equation*}
showing that it can be controlled with a cost 
$C_{null}\leq C_{ob}$.\\
\end{proof}

\begin{remark}\label{rk:diffproof}
Note that, as a novelty, the Riesz Representation Theorem is used in $L^2((0,T)\times G)$ instead of in $L^2(G)$ with the 
Hermitian product given by:
$$(z_T,\tilde z_T)\mapsto \int_0^T\int_{G_0}\mathbb E(z(t,x;\alpha;z_T))\overline{\mathbb E(z(t,x;\alpha; \tilde z_T))}dxdt.$$
Indeed, it is far from obvious that $\displaystyle\int_0^T\int_{G_0} |\mathbb E(z(t,x;\alpha;z_T))|^{2}dxdt=0$ implies that
$z_T=0$. 
\end{remark}

With regard to exact controllability, we have the following result, proved in \cite[Theorem A.1]{lu2016averaged}:
\begin{proposition}\label{prop:exavcont}
Let $T>0$. System \eqref{sys:Schr} is exactly controllable in average if and only if  system \eqref{sys:adjSchr} is exactly observable in average. Furthermore, the optimal controllability constant and the optimal observability constant are related by:
 \begin{eqnarray*}
 	C_{ex}=C_{exob}. 
 \end{eqnarray*}
\end{proposition}

\begin{remark}
It is equivalent to prove the averaged null (resp. exact) observability of \eqref{sys:adjSchr} as to prove that there is $C>0$ 
such that  the solutions of:
\begin{equation}\label{sys:adjSchrtimerev}
\begin{cases}
\partial_t z +\alpha \mathrm{i} \Delta z=0,& \mbox{ on }Q_T,\\
z=0, & \mbox{ on }\Sigma_T,\\
z(0,\cdot)=z_0, & \mbox{ in } G,
\end{cases}
\end{equation}
satisfy for all $z_0\in \mathcal{H}$ that:
\begin{equation}\label{est:obstimegood}
\left\|\mathbb{E}(z(T,\cdot;\alpha;z_0))\right\|\leq C \left\|\mathbb{E}(z(\cdot,\cdot;\alpha;z_{0}))\right\|_{L^2((0,T)\times G_0)}\end{equation}
\begin{equation}\label{est:obsexact}
\left(\mbox{resp.}\; \left\|z_0\right\|\leq C\left\|\mathbb{E}(z(\cdot,\cdot;\alpha;z_{0}))\right\|_{L^2((0,T)\times G_0)} \right).
\end{equation}
\end{remark}

In order to work with this system, the Fourier decomposition will be key.\\ 
Consider $\Lambda_\lambda \overset{\text{def}}{=} \{n:\lambda_n\leq \lambda\}$ for any $\lambda > 0$, and we denote by $\mathcal{P}_\lambda$ (resp. $\mathcal{P}^\perp_\lambda$) the orthogonal projection of $\mathcal{H}$ onto $\langle e_n \rangle_{n \in \Lambda_\lambda}$ (resp. $\langle e_n \rangle_{n \in \Lambda_\lambda}^\perp$). Then, using the orthogonality, we have that the average of \eqref{sys:adjSchrtimerev} is given by:
\begin{equation}\label{eq:Fourdec} \mathbb{E}(z(t, \cdot;\alpha;z_0))=\sum_{n\in\mathbb N}\langle z_0,e_n\rangle\varphi_{\alpha}(\lambda_{n}t)e_n,
\end{equation}
where $\varphi_{\alpha}$ is given in \eqref{def:charfunc}. 
\subsection{Main Hypothesis \texorpdfstring{\HypH}{HypH}} \label{Sec:HYP} 
We assume that the probability distribution $\mu_{\alpha}$ of $\alpha$ is given such that the Hypothesis \HypH is satisfied.
One consequence of Hypothesis \HypH is that $\mu_{\alpha}$ induces an absolutely continuous probability distribution with a probability kernel $\rho_{\alpha}\in C^\infty(\mathbb{R})$.
\begin{proposition} \label{Pr:absolutely continuous}
	Assume that Hypothesis \HypH holds. Then, $\mu_{\alpha}$ is an absolutely continuous 
	probability distribution 
	with a kernel $\rho_{\alpha}\in C^\infty(\mathbb{R})$.
\end{proposition}
\begin{proof}
	Let us recall the elementary properties of the \textbf{CF} of $\alpha$: 
	\begin{itemize}
		\item $\varphi_{\alpha}$ is continuous in $\mathbb{R}$,
		\item For all $s\in\mathbb{R}$, $|\varphi_{\alpha}(-s)|=|\varphi_{\alpha}(s)|$ and $\varphi_{\alpha}(0)=1$.
	\end{itemize}
	Let us consider the function 
	\[\rho_{\alpha}(x)=\frac{1}{2\pi} \int_{-\infty}^\infty e^{-\mathrm{i}sx}\vp_\alpha(s)ds.\]
	Using the bijectivity of the Fourier transform (see \cite[Section 3.3]{strichartz1994guide}), we have $\rho_{\alpha}(x)dx=d\mu_{\alpha}(x)$ as tempered distributions (and, in particular, as signed measures).
	To prove $\rho_{\alpha}\in C^{\infty}(\mathbb{R})$, we differentiate under the integral sign.
	Using \eqref{eq:hypSch}, we can see that $\varphi_\alpha$ decays faster than any polynomial. Taking $t_2=T_0$ and $t_1=0$, we obtain that for any real $s$ such that $|s|$ is sufficiently large:
	\[|\vp_\alpha(s)|=\left|\vp_{\alpha}\left(\frac{|s|}{T_0}T_0\right)\right|\leq 
	e^{-c|s|^rT_{0}^{\theta-r}}.
	\]
	Hence, for all $x, s \in\mathbb{R}$ and $n\in\mathbb{N}$, 
	$\left|\frac{\partial^{n}}{\partial x^n}(e^{-\mathrm{i}sx}\vp_\alpha(s))\right|\leq |s^n\varphi_{\alpha}(s)|$ and $s\mapsto |s^n\varphi_{\alpha}(s)|\in L^{1}(\mathbb{R})$. Thus, we deduce $\rho_{\alpha}\in C^\infty(\mathbb R)$.
\end{proof}
\par
A second remark is that Hypothesis \HypH is invariant by linear transformation of a random variable and sum of independent random variables:
\begin{remark}
	Let $\beta=a\alpha + b$ be the linear transformation of a random variable $\alpha$, where $a, b\in\mathbb{R}$ and $a\neq 0$. The \textbf{CF} of $\beta$ is given by 
	\[\varphi_\beta(s)=e^{\mathrm{i}sb}\varphi_{\alpha}(as), \quad s\in\mathbb{R}. \]
	As a result, Hypothesis \HypH is invariant by linear transformation of a random variable. It is also invariant to the sum of independent random variables, due to
	\[ \varphi_{\alpha}(s)=\prod_{k=1}^{n}\varphi_{\alpha_{k}}(s), \quad s\in\mathbb{R},\]
	where $\alpha=\displaystyle\sum_{k=1}^{n}\alpha_{k}$ and $\alpha_1, \cdots, \alpha_{n}$ are independent random variables.
\end{remark}
\par
Now, to provide examples, we are going to focus on absolutely continuous random variables.
\begin{example}
	A random variable $\alpha$ with normal distribution has \textbf{PDF} $\rho_{\alpha}$ given by
	\[ \rho_{\alpha}(\xi)\overset{\text{def}}{=}\frac{1}{\sqrt{2\pi \sigma^{2}}}e^{\frac{-(\xi-\mu)^{2}}{2\sigma^{2}}}, \quad \xi\in\mathbb{R},\]
	where $\mu\in\mathbb{R}$ (mean) and $\sigma^{2}>0$ (variance) and its \textbf{CF} is given by 
	\[ \varphi_{\alpha}(s)=e^{\mathrm{i}\mu s -\frac{\sigma^{2}s^{2}}{2}}, \quad s\in\mathbb{R}\]
	satisfies the estimate \eqref{eq:hypSch} for $c=\frac{\sigma^{2}}{2}$ and $r=\theta=2$ (as $t_2^2-t_1^2=(t_2-t_1)(t_2+t_1)\geq(t_2-t_1)^2$).
\end{example}
\begin{example}
	A random variable $\alpha$ with Cauchy distribution has \textbf{PDF} $\rho_{\alpha}$ given by
	\[ \rho_{\alpha}(\xi)\overset{\text{def}}{=}\frac{1}{\pi\gamma\left[ 1+\left(\frac{\xi-x_0}{\gamma}\right)^{2}\right]}, \quad \xi\in\mathbb{R},\]
	where $x_0\in\mathbb{R}$ (location) and $\gamma>0$ (scale) and its \textbf{CF} is given by 
	\[ \varphi_{\alpha}(s)=e^{\mathrm{i}x_0 s-\gamma|s|}, \quad s\in\mathbb{R}\]
	satisfies the estimate \eqref{eq:hypSch} for $c=\gamma$ and $r=\theta=1$.
\end{example}
\par 
We can also provide an interesting class of random variables that meet the Hypothesis \HypH, specifically that of stable distributions as defined in \cite{klebanov2003heavy}.
\begin{definition}
	A random variable $\alpha$ has stable distribution if and only if the \textbf{CF} of $\alpha$ has the form:
	\begin{eqnarray*}
		\varphi_{\alpha}(s)\overset{\text{def}}{=}\exp\left[\mathrm{i}\mu s-c|s|^{r}\left(1+\mathrm{i}\beta\text{sign}(s)\Phi(s,r)\right)\right],
	\end{eqnarray*}
	where $r\in (0,2]$ is the stability parameter, $\mu\in\mathbb{R}$ is a shift parameter, $\beta\in [-1,1]$ called the skewness parameter, $c>0$, $\text{sign}(s)$ is the sign of $s$ and 
	\begin{eqnarray*}
		\Phi(s,r)\overset{\text{def}}{=}\begin{cases}
			\tan\left(\frac{\pi r}{2}\right)\quad & r\neq 1,\\
			-\frac{2}{\pi}\log(|s|) \quad & r= 1.
		\end{cases}
	\end{eqnarray*}
\end{definition}
\begin{remark}
	Note that, if $(r,\beta)=(2,0)$ and $(r,\beta)=(1,0)$ we find the normal distribution and the Cauchy distribution, respectively. All stable distributions are absolutely continuous and have continuous densities; see Theorem 2.3 of \cite{klebanov2003heavy}. 
\end{remark}
\begin{remark}
	For $r \in\left(1/2, 2\right]$, by an elementary proof we can show that any random variable with stable distribution satisfies Hypothesis \HypH.  In fact, if $r\in(1/2,1]$, the function $t\mapsto t^{\frac{1}{r}}$ is Lipschitz on $[0,T^r]$, so, if $t_1,t_2\in[0,T]$ such that $t_1<t_2$, we have: $$t_2-t_1=(t_2^r)^{1/r}-(t_1^r)^{1/r}\leq L_r(t_2^r-t_1^r),$$ for some $L_r>0$. Hence,
    \[|\varphi_\alpha(\lambda t_2)|=e^{-c\lambda^rt_2^r}\leq e^{-\frac{c\lambda^r}{L_r}(t_2-t_1)}|\varphi_\alpha(\lambda t_1)|.
    \]
    When $r\in (1,2]$, for all $t_1,t_2\in[0,T]$ such that $t_1<t_2$ we have that: $$(t_2 -t_1)^r\leq t_2^r-t_1^r,$$ as $x\mapsto (t_2-x)^r -t_2^r+ x^r$ is convex in $[0,t_2]$ and null on $0$ and $t_2$. Hence,
    \[|\varphi_\alpha(\lambda t_2)|=e^{-c\lambda^rt_2^r}\leq e^{-c\lambda^r(t_2-t_1)^r}|\varphi_\alpha(\lambda t_1)|.
    \]
\end{remark}

\section{Averaged null controllability} \label{Sec3}
In this section, we present our main result regarding the averaged null 
controllability of the Schrödinger system \eqref{sys:Schr}.
\begin{theorem} \label{Main result}
	Let $G \subset \mathbb{R}^{d}$ be a Lipschitz locally star-shaped domain, $G_0 \subset G$ be a subset of strictly positive measure, $T>0$ and $\alpha$ a random variable whose characteristic function satisfies Hypothesis \HypH for certain parameters $c>0$, $r>\frac{1}{2},\; \theta>0$ and $T_0\in (0,T]$. Then, there are $C>0$ and $T^{\prime}\in (0, T_0]$ such that, for all $T_1\in (0, T^\prime]$, 
    system \eqref{sys:Schr} is null controllable in
	average in time $T_1$,
    and the cost of null controllability satisfies:
	\[ C_{null}(G,G_0, \alpha, T_1) \leq C e^{C T_1^{-\theta(2r-1)^{-1}}}.\]
\end{theorem}
As explained in the Introduction, we obtain an observability inequality with a spectral approach as in \cite{miller2010observability}. In order to derive an elliptic observability estimate for general subsets of strictly positive measure, by exploiting results from the literature, we begin with the following technical lemma.
\begin{lemma}
	Let $G \subset \mathbb{R}^d$ be a Lipschitz domain, and let $G_0 \subset G$ be a subset of strictly positive measure. Then there exist $x_0\in G$ and $r_0 \in (0,1]$ such that
	\[
	|G_0 \cap B(x_0, r_0)| > 0 \quad \text{and} \quad B(x_0, 4r_0) \subset G.
	\]
\end{lemma}
We provide a short proof for the sake of completeness, even if it is a standard result:
\begin{proof}
Let us define $G_\delta=\{x\in G:d(x,\partial G)\geq\delta\}$, where $d(x,\partial G)$ is the Euclidean distance from $x$ to $\partial G$. For $\delta$ small enough (and, in particular, $\delta\leq4$), we have $|G_0\cap G_\delta|>0$. By compactness of $G_\delta$, there are $x_1,\ldots,x_n\in G_\delta$ such that $G_\delta\subset\bigcup_{i=1}^nB(x_i,\delta/4)$. Clearly, one of such balls must satisfy the required result.
\end{proof}


Now, applying Theorem 3 and Theorem 5 of \cite{apraiz2014observability} for the observation set $G_0 \cap B(x_0, r_0)$, we obtain the following elliptic observability:
\begin{lemma}\label{lm:obsellG0}
	Let $G\subset \mathbb{R}^{d}$ be a Lipschitz locally star-shaped domain, $G_0\subset G$ be a subset of strictly positive measure, and $\{e_n\}$ be the orthonormal eigenfunctions of the Dirichlet Laplacian.
    Then, there exists a constant
	$C>0$ such that for all $\lambda>0$ and 
	$\{c_n\}\subset\bb R$:
	\begin{equation}\label{est:seqobs}
		\left(\sum_{n\in\Lambda_\lambda}|c_n|^2\right)^{1/2}\leq C
		e^{C\sqrt \lambda}\left\|\sum_{n\in\Lambda_\lambda} c_ne_n
		\right\|_{L^2(G_0)}.
	\end{equation}
\end{lemma}\noindent
This result is an improved version of \cite[Theorem 1.2]{lu2013lower},
which was first proved for more regular cases in \cite{lebeau1995controle}. 
\par 
In order to estimate the cost of the control, it suffices to
prove the analogue of \cite[Lemma 2.3]{miller2010observability}:
\begin{lemma}\label{lm:ineqprac}
	Let $G\subset\bb R^d$ be a domain, $G_0\subset G$ be a subset of strictly positive measure, 
	$T_0,\beta,\gamma_1,\gamma_2,f_0,g_0>0$ satisfying 
	$\gamma_1<\gamma_2$. Suppose that we have
	for all $z_0 \in \mathcal{H}$ and all $t_1,t_2\in (0,T_0]$ satisfying $t_1<t_2$
	the  inequality:
	\begin{equation}\label{eq:recriwthfg}
		f(t_2-t_1)\|\mathbb{E}(z(t_2,\cdot;\alpha;z_{0}))\|^2-g(t_2-t_1)
		\|\mathbb{E}(z(t_1,\cdot;\alpha;z_{0}))\|^2
		\leq \int_{t_1}^{t_2} \int_{G_0}|\mathbb{E}(z(\tau,x;\alpha;z_{0}))|^2dxd\tau, 
	\end{equation}
	where $f$ and $g$ satisfy $f(s)\geq f_0\exp(-2/(\gamma_2s)^\beta)$, $\dis\lim_{s\to0^+}f(s)=0$ and $g(s)\leq g_0\exp(-2/(\gamma_1s)^\beta)$.
	Then, for any $\gamma\in(0,\gamma_2-\gamma_1)$ there is $T'\in (0,T_0]$ such that
	for all $T_1\in (0,T']$ and $z_0\in \mathcal{H}$:
	\[\left\|\mathbb{E}(z(T_1,\cdot;\alpha;z_{0}))\right\|\leq \sqrt{f_0^{-1}}
	\exp(1/(\gamma T_1)^\beta)
	\|\mathbb{E}(z(\cdot,\cdot;\alpha;z_{0}))\|_{L^2((0,T_1)\times G_0)}.\]
\end{lemma}
\begin{proof}It suffices to prove that for any $t_{1,k},t_{2,k}\in(0,T_1]$ with 
$t_{1,k}<t_{2,k}$:
\[f(t_{2,k}-t_{1,k})\|\mathbb{E}(z(t_{2,k},\cdot;\alpha;z_{0}))\|^2-f(q(t_{2,k}-t_{1,k})) \|\mathbb{E}(z(t_{1,k},\cdot;\alpha;z_{0}))\|^2
\leq \int_{t_{1,k}}^{t_{2,k}} \int_{G_0}|\mathbb{E}(z(\tau,x;\alpha;z_{0}))|^2dxd\tau,
\]
for  $q=1-\frac{\gamma}{\gamma_2}$ and $T_1$ small enough. Afterwards, we just have to apply a telescopic series with $t_{2,k}=T_1 q^k$ and $t_{1,k}=T_1 q^{k+1}$, noting that $f(q(t_{2,k}-t_{1,k})) \to 0$ as $k \to \infty$. 

For that, it suffices to obtain that:
\[g(s)\leq f(qs),\]
for $s>0$ small enough.
For that, it suffices to show that: 
\[\frac{g_0}{f_0}\exp\left(\frac{-2}{(\gamma_1s)^\beta}
+\frac{2}{(\gamma_2qs)^\beta}\right)\leq 1
\]
which clearly holds as $q>\frac{\gamma_1}{\gamma_2}$ and as $s$ is small enough.
\end{proof}

Firstly, let us show the decay of the solutions: 
\begin{lemma}\label{rk:decreasenorm}
	Let $G\subset\mathbb R^d$ be a domain, and
    $\alpha$ be a random variable satisfying 
	Hypothesis \HypH. Then,
	\begin{itemize}
		\item $|\vp_\alpha|$ is a strictly decreasing function in $[0,\infty)$.
		\item $t\mapsto \mathcal \| \mathbb{E}(z(t,\cdot;\alpha;\mathcal P_\lambda z_{0}))\|$, $t\mapsto \| \mathbb{E}(z(t,\cdot;\alpha;\mathcal P_\lambda^\perp z_{0}))\|$,
        and $t\mapsto \mathcal \|\mathbb{E}(z(t,\cdot;\alpha;z_{0}))\|$ are decreasing functions in $[0,T]$ for all $\lambda\geq \lambda_0$ and $z_{0}\in \mathcal{H}$. We recall that $\lambda_0>0$ is the first eigenvalue of $-\Delta$ under Dirichlet boundary condition.
	\end{itemize} 
\end{lemma}
\begin{proof}
	Let $s_1,s_2\in[0,\infty)$ with $s_1<s_2$ and $\lambda$ large enough so that $\frac{s_2}{\lambda}\leq T_0$. Then, if $t_1=\frac{s_1}{\lambda}$ and $t_2=\frac{s_2}{\lambda}$, we have from \eqref{eq:hypSch} that:
	\[|\vp_\alpha(s_2)|=|\vp_\alpha(\lambda t_2)|
	\leq e^{-c\lambda^r(t_2-t_1)^{\theta}}|\vp_\alpha(\lambda t_1)|
	<|\vp_\alpha(\lambda t_1)|=|\vp_\alpha(s_1)|.
	\] 
    Thus, $|\vp_\alpha|$ is strictly decreasing in $[0,\infty)$.
    
	Let $\lambda\geq \lambda_0$ and $z_{0}\in \mathcal{H}$. Applying Parseval's identity to \eqref{eq:Fourdec}, we obtain:
	\begin{eqnarray*}
			&&\|\mathbb{E}(z(t,\cdot;\alpha;\mathcal P_\lambda z_{0}))\|^{2}=\sum_{n\in\mathbb \varLambda_\lambda}|\langle z_0,e_n\rangle|^{2}|\varphi_{\alpha}(\lambda_{n}t)|^{2},\\
            && \| \mathbb{E}(z(t,\cdot;\alpha;\mathcal P_\lambda^\perp z_{0}))\|^2=\sum_{n\in\mathbb N\setminus \varLambda_\lambda}|\langle z_0,e_n\rangle|^{2}|\varphi_{\alpha}(\lambda_{n}t)|^{2},\\
			&&\|\mathbb{E}(z(t,\cdot;\alpha;z_{0}))\|^{2}=\sum_{n\in\mathbb N}|\langle z_0,e_n\rangle|^{2}|\varphi_{\alpha}(\lambda_{n}t)|^{2}.
		\end{eqnarray*}
	Thus, the second point arises from the fact that $|\varphi_\alpha|$ is decreasing.  
\end{proof}

Let us now prove Theorem \ref{Main result}.
\begin{proof}[Proof of Theorem \ref{Main result}]
	Let $t_1,t_2\in (0,T_0]$ such that $t_1<t_2$, $z_0 \in \mathcal{H}$ and $\lambda\geq\lambda_0$.
	We are going to prove:
	\begin{equation} \label{eq:interpineq}
		\begin{split}
			&C^{-1}e^{-C((t_2-t_1)^{-\frac{\theta}{2r-1}} +\sqrt{\lambda})}
			\|\mathbb{E}(z(t_2,\cdot;\alpha;z_{0}))\|^2
			-  e^{-\frac{c}{2^\theta}\lambda^r(t_2-t_1)^{\theta}}
			\|\mathbb{E}(z(t_1,\cdot;\alpha;z_{0}))\|^2\\
			&\leq 
			\int_{t_1}^{t_2}\int_{G_0}
			|\mathbb{E}(z(\tau, x;\alpha;z_{0}))|^2 dxd\tau,
		\end{split}
	\end{equation}
	where $C>0$ is large enough and $c>0$ as in \HypH, and then use Lemma \ref{lm:ineqprac}
	with the appropriate value of $\lambda$ (depending this 
	value on $t_1$ and $t_2$). 
	First, considering the orthogonal decomposition $\cal P_\lambda z_{0}\perp \cal P_\lambda^\perp z_{0}$ of $z_0$ induced by the Laplacian
    and Lemma \ref{rk:decreasenorm}	we have that:
	\begin{equation}\label{eq:decayut2}
    \begin{split}
		\|\mathbb{E}(z(t_2,\cdot;\alpha;z_{0}))\|^2
		&=\|\mathbb{E}(z(t_2,\cdot;\alpha;\mathcal P_\lambda z_{0}))\|^2+
        \|\mathbb{E}(z(t_2,\cdot;\alpha;\mathcal P_\lambda^\perp z_{0}))\|^2
        \\&\leq \frac{2}{t_2-t_1}\int_{\frac{t_1+t_2}{2}}^{t_2}\int_G
		(|\mathbb{E}(z(\tau, x;\alpha;\mathcal P_\lambda z_{0}))|^2+ 
		|\mathbb{E}(z(\tau, x;\alpha;\mathcal P^\perp_\lambda z_{0}))|^2)dxd\tau.
	\end{split}
    \end{equation}
	From Lemma \ref{lm:obsellG0}
	and that $\cal P_\lambda z_{0}=z_{0}-\cal P^\perp_\lambda z_{0}$ 
	we obtain that: 
	\begin{equation} \label{eq:localfull}
		\begin{split}
			\frac{2}{t_2-t_1}\int_{\frac{t_1+t_2}{2}}^{t_2}\int_G
			|\mathbb{E}(z(\tau, x;\alpha; \cal P_\lambda z_{0}))|^2 dxd\tau &\leq C
			\frac{e^{C\sqrt{\lambda}}}{t_2-t_1} \int_{\frac{t_1+t_2}{2}}^{t_2}\int_{G_0}
			|\mathbb{E}(z(\tau, x;\alpha; \cal P_\lambda z_0))|^2dxd\tau\\
			&\leq  \frac{Ce^{C\sqrt{\lambda}}}{t_2-t_1} 
			\int_{\frac{t_1+t_2}{2}}^{t_2}\int_{G_0}
			|\mathbb{E}(z(\tau, x;\alpha; z_{0}))|^2dxd\tau \\
			&+\frac{Ce^{C\sqrt{\lambda}}}{t_2-t_1} \int_{\frac{t_1+t_2}{2}}^{t_2}\int_G
			|\mathbb{E}(z(\tau, x;\alpha;\mathcal P_\lambda^\perp z_0))|^2 dxd\tau.
		\end{split}
	\end{equation}
	Moreover, from  the decay property 
	of Lemma \ref{rk:decreasenorm}, the spectral decomposition of the solution, and \eqref{eq:hypSch}
	we have that:
	\begin{equation} \label{eq:t2andt1}
		\begin{split}
			\frac{Ce^{C\sqrt{\lambda}}}{t_2-t_1}\int_{\frac{t_1+t_2}{2}}^{t_2}\int_G 
			|\mathbb{E}(z(\tau, x;\alpha; \cal P_\lambda^\perp z_{0}))|^2dxd\tau&\leq 
			Ce^{C\sqrt{\lambda}} \left\|\mathbb{E}(z((t_2+t_1)/2, \cdot ;\alpha;\cal P^\perp_{\lambda} z_{0}))\right\|^2\\
			&\leq Ce^{C\sqrt{\lambda}-\frac{c}{2^\theta}\lambda^r(t_2-t_1)^{\theta}}
			\|\mathbb{E}(z(t_1, \cdot ;\alpha;\cal P^\perp_{\lambda} z_{0}))\|^2.
		\end{split}
	\end{equation}
	Since $r>\frac{1}{2}$ and $\theta>0$, for 
    \begin{equation}\label{eq:defsigma}
    \sigma := \frac{\theta}{2r-1}>0,
    \end{equation}
    we obtain $(t_2-t_1)^{-1}\leq Ce^{C(t_2-t_1)^{-\sigma}}$.
	Thus, from \eqref{eq:decayut2}-\eqref{eq:t2andt1},  
	$2\leq C$
	and 
	$(t_2-t_1)^{-1}\leq Ce^{C(t_2-t_1)^{-\sigma}}$ (recall that $C>0$ is a sufficiently large constant) we obtain:
	\begin{equation*}
		\begin{split}
			\|\mathbb{E}(z(t_2, \cdot ;\alpha; z_{0}))\|^2
			&\leq \frac{Ce^{C\sqrt{\lambda}}}{t_2-t_1} 
			\int_{\frac{t_1+t_2}{2}}^{t_2}\int_{G_0}
			|\mathbb{E}(z(\tau, x ;\alpha; z_{0}))|^2dxd\tau \\
			&+ 
			Ce^{C\sqrt{\lambda}-\frac{c}{2^\theta}\lambda^r(t_2-t_1)^{\theta}}
			\|\mathbb{E}(z(t_1, \cdot ;\alpha; \cal P^\perp_{\lambda} z_{0}))\|^2\\
			& \leq	Ce^{C((t_2-t_1)^{-\sigma}+\sqrt{\lambda})}\int_{\frac{t_1+t_2}{2}}^{t_2}\int_{G_0}
			|\mathbb{E}(z(\tau, x ;\alpha; z_{0}))|^2dxd\tau \\
			&+ 
			Ce^{C\sqrt{\lambda}-\frac{c}{2^\theta}\lambda^r(t_2-t_1)^{\theta}}
			\|\mathbb{E}(z(t_1, \cdot ;\alpha; \cal P^\perp_{\lambda} z_{0}))\|^2.
		\end{split}
	\end{equation*}
        Now, multiplying the latter inequality by $C^{-1}e^{-C((t_2-t_1)^{-\sigma} +\sqrt{\lambda})}$ and using $e^{-C(t_2-t_1)^{-\sigma}}<1$, we obtain
	\begin{equation} \label{eq:obsalmost}
		\begin{split}
			&C^{-1}e^{-C((t_2-t_1)^{-\sigma} +\sqrt{\lambda})}\|
			\mathbb{E}(z(t_2,\cdot; \alpha; z_0))\|^2\\
			&\leq 
			\int_{t_1}^{t_2}\int_{G_0} |\mathbb{E}(z(\tau,x;\alpha; z_0))|^2dxd\tau
			+ e^{-\frac{c}{2^\theta}\lambda^r(t_2-t_1)^{\theta}} \|\mathbb{E}(z(t_1,\cdot; \alpha;
			\cal P^\perp_{\lambda} z_0))\|^2\\
			&\leq 
			\int_{t_1}^{t_2}\int_{G_0} |\mathbb{E}(z(\tau,x;\alpha; z_0))|^2dxd\tau
			+ e^{-\frac{c}{2^\theta}\lambda^r(t_2-t_1)^{\theta}} \|\mathbb{E}(z(t_1,\cdot; \alpha;
			z_0))\|^2,
		\end{split}
	\end{equation}
	which implies \eqref{eq:interpineq}.
	\paragraph{}
	We now define:
	\begin{equation}\label{def:lambdat1t2}
		\lambda(t_2,t_1)=\frak C(t_2-t_1)^{-2\sigma}, 
	\end{equation} 
	for $\frak C\geq\lambda_0$ a positive constant sufficiently large.
	If we take in \eqref{eq:interpineq}
	$\lambda$ given by \eqref{def:lambdat1t2}, using \eqref{eq:defsigma}
	we obtain \eqref{eq:recriwthfg} for the functions:
	\begin{equation*}
		f(s)= C^{-1}\exp\left(-C\left(1+\frak C^{1/2}\right) s^{-\sigma}
		\right) \quad \text{and}\quad
		g(s)=\exp\left(-\frac{c}{2^\theta}\frak C^{r} s^{-2r\sigma+\theta}\right)=\exp\left(-\frac{c}{2^\theta}\frak C^{r} s^{-\sigma}\right).
	\end{equation*}
    Clearly, $\dis\lim_{s\to0^+}f(s)=0$
	In addition,  we have for $\sigma$ given in \eqref{eq:defsigma} and all $s\in(0,1)$  that:
	\[f(s)\geq C^{-1}\exp\left(-2C \frak C^{1/2}s^{-\sigma}
	\right)\quad \text{and}\quad g(s)\leq \exp\left(-\frac{c}{2^\theta}\frak C^{r} s^{-\sigma}\right).\]
	Moreover, since $r>\frac{1}{2}$, the functions $f$ and $g$ satisfy the hypotheses of Lemma \ref{lm:ineqprac} for $\beta=\sigma$,
	$\gamma_1= \left(\frac{c\frak C^r}{2^{1+\theta}}\right)^{-1/\beta}$ and 
	$\gamma_2= (C\frak C^{1/2})^{-1/\beta}$ by taking $\frak C$ large enough. Consequently, we end the proof using Lemma \ref{lm:ineqprac}.
\end{proof}

\section{Lack of controllability} \label{Sec4}
In this section, we present other contributions to the literature
regarding the lack of controllability
of the system \eqref{sys:Schr} for absolutely continuous random variables, as well as the lack of simultaneous null controllability.
\subsection{Lack of exact averaged controllability}
\paragraph{}
In contrast to the literature, a general result can be derived for the lack of exact controllability in $L^{2}(G)$ with controls in $L^2(0,T;L^2(G_0))$ for absolutely continuous random variables. This includes, in particular, uniform, exponential, Laplace, and chi-squared random variables. However, it should be noted that for these random variables, exact controllability has been proven in \cite[Theorem 4.3]{lu2016averaged}, either when the control is in $L^2(0,T;H^{-2}(G_0))$ or within spaces of the type $H^{2k}(G)\cap H_0^1(G)$ for some $k\geq1$. In particular, we show that $L^2$ controllability and observability properties (such as those proved in \cite{machtyngier1994exact} and \cite{phung2001observability}) of the Schrödinger equation are not inherited when we consider its average. 
\par 
Let us recall the Riemann-Lebesgue lemma, whose proof is in  \cite[Section 7.1]{strichartz1994guide}
\begin{lemma}[Riemann-Lebesgue lemma]\label{lm:RL}
	Let $f\in L^1(\mathbb{R})$. Then,
	\[\lim_{z\to\infty}\int_{-\infty}^{\infty}e^{-\mathrm{i}zs} f(s)ds=
	\lim_{z\to-\infty}\int_{-\infty}^{\infty} e^{-\mathrm{i}zs} f(s)ds=0.
	\]
\end{lemma}
The following result is a contribution to the literature.
\begin{theorem} \label{tm:noexcon}
	Let $\alpha$ be an absolutely continuous random variable in $(\Omega, \mathcal{F}, \mathbb{P})$, $G\subset\mathbb R^d$ be a Lipschitz domain, and $G_0\subset G$ be a subset of strictly positive measure. Then,  system \eqref{sys:Schr} is not exactly controllable in average in the space $L^2(G)$ with controls acting in $L^2(0,T;L^2(G_0))$.
\end{theorem}
\begin{proof}
	The proof is based on the Riemann-Lebesgue lemma (see Lemma \ref{lm:RL}). Let us suppose that there is $C>0$ such that
	the observability inequality \eqref{est:obsexact} holds and 
	let us obtain a contradiction. Consider as initial values $z_0$ the eigenfunction $e_n$ for $n\in\mathbb{N}$. In this case,
	\[\|e_n\|=1. 
	\]
	Then, assuming \eqref{est:obsexact}, for such initial value we obtain:
		\begin{align} 
			1 &\leq C\int_0^T\int_{G_0}  |\mathbb{E}(z(t,x;\alpha; e_n))|^2 dxdt \nonumber\\
			&\leq C\int_0^T\int_{G}  |\mathbb{E}(z(t,x;\alpha;e_n))|^2 dxdt \nonumber\\
			&=C\int_0^T\int_{G}  |\varphi_{\alpha}(\lambda_{n}t)e_n(x)|^2 dxdt  \nonumber\\
			&= C\int_0^T|\varphi_{\alpha}(\lambda_{n}t)|^2 dt. \label{est:eq}
		\end{align}
	Since $\alpha$ is an absolutely continuous random variable, then its \textbf{PDF} $\rho_{\alpha}\in L^{1}(\mathbb{R})$ and its \textbf{CF} is given by:
	\[\varphi_{\alpha}(\lambda_{n}t)=\int_{-\infty}^{\infty}e^{\mathrm{i}(\lambda_{n}t)\xi}\rho_{\alpha}(\xi)d\xi.\]
	Using the Riemann-Lebesgue lemma (see Lemma \ref{lm:RL}), for all $t\in (0,T]$, $\varphi_{\alpha}(\lambda_{n}t)$ converges to $0$ as $n\rightarrow \infty$. Moreover, $|\varphi_{\alpha}(\lambda_{n}t)|\leq 1$. Thus, the Dominated Convergence Theorem shows that the last term in \eqref{est:eq} converges to $0$ as $n\rightarrow \infty$. Therefore, an estimate such as \eqref{est:obsexact} is not possible, and the result follows from Proposition \ref{prop:exavcont}.
\end{proof}
As a consequence of Theorem \ref{tm:noexcon} and Proposition \ref{Pr:absolutely continuous}, we have the following corollary:
\begin{corollary}
	Let $\alpha$ be a random variable that satisfies Hypothesis \HypH,  $G\subset\mathbb R^d$ a Lipschitz domain, and $G_0\subset G$ a subset of strictly positive measure. Then, the system \eqref{sys:Schr} is not exactly controllable in average in the space $L^2(G)$ with controls acting in $L^2(0,T;L^2(G_0))$.
\end{corollary}

\subsection{Lack of simultaneous null controllability}
Now, we will show that we do not have simultaneous null controllability for any absolutely continuous random variable. More precisely, the realizations of such a notion are negligible, which is an immediate consequence of the following theorem:
\begin{theorem}\label{thm:nonsimcon}
	Let $G\subset\mathbb R^d$ be a Lipschitz domain, $G_0\subset G$ be a non-empty subset of strictly positive measure, $T>0$, $y_0\in L^2(G)\setminus\{0\}$ and
	$u\in L^2((0,T)\times G_0)$. Then, the set:
	\begin{eqnarray*}
		\mathcal{A}\overset{\text{def}}{=}\left\{\xi\in\mathbb R\;:\; y(T,\cdot;\xi;y_0;u)=0\right\}.
	\end{eqnarray*}
	is finite.
\end{theorem}
In the proof of this theorem, we will use the following result stated in \cite[Theorem 7.2.1]{strichartz1994guide}:
\begin{lemma}[Paley-Wiener Theorem]\label{lm:analycond}
	Let $T>0$ and $f\in L^2(0,T)$. Then 
	\[z\mapsto \int_0^Te^{-\mathrm{i}zs} f(s)ds
	\]
	is an analytic function.
\end{lemma}

\begin{proof}[Proof of Theorem \ref{thm:nonsimcon}]
	Using Duhamel formula, we have that:
	\begin{eqnarray*}
		\mathcal{A}=\left\{ \xi\in\mathbb R\;:\; \int_0^Te^{-\mathrm{i}s\xi\Delta}\mathds{1}_{G_0}u(s,\cdot)ds=-y_0\right\}. 
	\end{eqnarray*}
	Developing $e^{-\mathrm{i}s\xi\Delta}\mathds{1}_{G_0}u(s,\cdot)$ as a Fourier series, we have that:
	\begin{equation*}
		\int_0^Te^{-\mathrm{i}s\xi\Delta}\mathds{1}_{G_0}u(s,\cdot)ds
		=\sum_{n\in\mathbb{N}} \left[\left(\int_0^Te^{-\mathrm{i}\lambda_n s\xi}
		f_n(s) ds \right) e_n \right],
	\end{equation*}
	for
	\[f_n(s)=\langle\mathds{1}_{G_0}u(s,\cdot),e_n \rangle.
	\]
	It follows that:
	\begin{eqnarray*}
		\mathcal{A}=\bigcap_{n\in\mathbb{N}}\left\{ \xi\in\mathbb R\;:\; \int_0^Te^{-\mathrm{i}\lambda_n s\xi}
		f_n(s) ds=-\langle y_0, e_n\rangle \right\} \overset{\text{def}}{=} \bigcap_{n\in\mathbb{N}}\mathcal{A}_{n}. 
	\end{eqnarray*}
	So, it is enough to prove that there exists $n\in\mathbb{N}$ such that $\mathcal{A}_{n}$ is finite.
	Let us pick $n$ such that 
    \[\langle y_0,e_n\rangle\neq0,\]
    which implies that $f_n(s)\neq0$. 
	Since $y_0\neq0$ such $n$ exists. Now, let us suppose for the sake of contradiction that $\mathcal A_n$ is infinite. Then, we have one of the following cases:
	\begin{itemize}
		\item $\mathcal A_n$ is unbounded. In that case, there is a sequence $\left\{\xi_m\right\}\subset\mathcal A_n$ such that either $\xi_m\to\infty$ or $\xi_m\to-\infty$. In that case, as $f_n\in L^2(0,T)\subset L^1(0,T)$, by the Riemann-Lebesgue lemma (see Lemma \ref{lm:RL}):
		\[\lim_{m\to\infty} \int_0^Te^{-\mathrm{i}\lambda_n s\xi_m}
		f_n(s) ds =0,\]
		which contradicts \[\int_0^Te^{-\mathrm{i}\lambda_n s\xi_m}
		f_n(s) ds=-\langle y_0, e_n\rangle \ \ \forall m\in\mathbb N.\]Thus, this case is not possible.
		\item $\mathcal A_n$ is bounded. By continuity of $\xi\mapsto \displaystyle\int_0^Te^{-\mathrm{i}\lambda_n s\xi}
		f_n(s) ds$ we can prove that $\mathcal A_n$ is also close. 
		Then, by Bolzano-Weirestrass, there is a sequence $\left\{\xi_m \right\}\subset \mathcal A_n$ of distinct values and $\tilde\xi\in \mathcal A_n$ such that $\xi_m\to \tilde\xi$ and $\xi_m\neq \tilde\xi$. However, note that 
		$$\phi_n(\xi) \overset{\text{def}}{=} \int_0^Te^{-\mathrm{i}\lambda_n s\xi}f_n(s) ds$$
		is analytic in $\mathbb R$, which follows from Lemma \ref{lm:analycond}.
		Then, $\phi_n + \langle y_0, e_n\rangle$ has $\tilde\xi$ as an accumulation point of its zeros, and from there one can obtain that the Taylor series around $\tilde \xi$ is constant (the first non-constant term of the series would prevent the accumulation point), and thus $\phi_n$ would be constant. But $\phi_n$ is essentially the Fourier transform of a non-null function (of $\mathds 1_{(0,T)}f_n$)
		in $L^1(\mathbb R)$, so it is absurd that $\phi_n$ is constant (the inverse Fourier transform of constant functions are Dirac masses). Thus, this case is also not possible. 
	\end{itemize}
	In conclusion, it is absurd that $\mathcal A_n$ is infinite, so it must be finite.
\end{proof}
\section{Some numerical results and experiments} \label{Sec5}
The aim of this section is to illustrate numerically the averaged null controllability of system \eqref{sys:Schr}.
\subsection{Algorithm for calculating HUM controls}
In this section, we look at a numerical algorithm to calculate HUM controls, which provides a control with minimal $L^2$-norm. 
We refer to \cite{boyer2013penalised, glowinskiexact,lazar2022control} for more details on this method for parabolic equations and other systems, such as the wave equation and the Stokes system. In \cite{ boutaayamou2025null}, this method is applied to a heat equation coupled with an ordinary differential equation. In the following, we will adapt this method to the Schrödinger equation \eqref{sys:Schr} in the average sense. For numerical results on average controllability and ensemble controllability of finite-dimensional systems, we refer the reader to \cite{lazar2022control}.
\par 
Let $y_{0}\in \mathcal{H}$ be an initial datum to be controlled and we define the cost functional by
\begin{eqnarray*}
	\mathcal{J}(z_{T}) \overset{\text{def}}{=} \frac{1}{2}\int_{0}^{T}\int_{G_{0}}\left|\mathbb{E}(z(t,x;\alpha; z_T))\right|^{2}dxdt +\left\langle y_{0},\mathbb{E}(z(0,\cdot;\alpha; z_T))\right\rangle,
\end{eqnarray*}
where $z(\cdot,\cdot;\alpha; z_T)$ denotes the solution of \eqref{sys:adjSchr} associated with the final data $z_{T}$. 
The minimizer $\tilde{z}_{T}$ of $\mathcal{J}$ is characterized by the Euler-Lagrange equation:
\begin{eqnarray}
  \int_{0}^{T}\int_{G_{0}}\mathbb{E}(z(t,x;\alpha; z_T))\overline{\mathbb{E}(z(t,x;\alpha; \tilde{z}_T))}dxdt +\left\langle y_{0},\mathbb{E}(z(0,\cdot;\alpha; z_T))\right\rangle=0, \label{EL}
\end{eqnarray}
for all $z_T\in \mathcal{H}$.
Consequently, based on the proof of Proposition \ref{Pr: obser_null cont}, we can choose the following quantity as a control:
\begin{eqnarray*}
    u(t,x)=\mathds{1}_{G_{0}}\mathbb{E}(z(t,x;\alpha; \tilde{z}_{T})).
\end{eqnarray*}
To express the Euler-Lagrange equation \eqref{EL} as a linear equation whose minimizer $\tilde{z}_{T}$ is its solution,  
let us now define the linear operator $\Lambda$, usually referred to as the Gramian operator, as follows
\begin{eqnarray*}
	\Lambda(z_T) \overset{\text{def}}{=} \mathbb{E}(y(T,\cdot;\alpha;0; \mathbb{E}(z(\cdot,\cdot;\alpha; z_T)))), 
\end{eqnarray*}
where $y(\cdot,\cdot;\alpha;0; \mathbb{E}(z(\cdot,\cdot;\alpha; z_T)))$ denotes the solution of \eqref{sys:Schr} associated with the initial data $y_{0}=0$ and $u=\mathbb{E}(z(\cdot,\cdot;\alpha; z_T))$.
The duality argument yields the following:
\begin{eqnarray}
	\int_{0}^{T}\int_{G_{0}}\mathbb{E}(z(t,x;\alpha; z_T))\overline{\mathbb{E}(z(t,x;\alpha; \tilde{z}_T))}dxdt=\langle \Lambda(\tilde{z}_{T}),z_{T}\rangle. \label{EL1}
\end{eqnarray}
Again applying the duality argument, we obtain:
\begin{eqnarray}
	\left\langle y_{0},\mathbb{E}(z(0,\cdot;\alpha; z_T)) \right\rangle=\left\langle \mathbb{E}(y(T,\cdot;\alpha; y_0; 0)), z_{T} \right\rangle, \label{EL2}
\end{eqnarray}
where $y(\cdot,\cdot;\alpha; y_0; 0)$ denotes the solution of \eqref{sys:Schr} associated with the initial data $y_0$ and $u=0$.
By injecting \eqref{EL1} and \eqref{EL2} into \eqref{EL},  we obtain the following linear equation:
\begin{eqnarray*}
	\Lambda(\tilde{z}_{T})=-\mathbb{E}(y(T,\cdot;\alpha; y_0; 0)).
\end{eqnarray*}
To resolve this operator equation, we propose the Conjugate Gradient method in Algorithm \ref{algo1}, which is an efficient algorithm for solving linear systems.
\begin{algorithm}
	\caption{Conjugate Gradient Method (CG)} \label{algo1}
	\begin{algorithmic}[1]
		\State \textbf{Input:} An initial state to be controlled \( y_0\in \mathcal{H} \), linear operator \( \Lambda \), initial guess \( z_0 \), tolerance \( \text{tol} \), maximum iterations \(k_{\text{max}}\) 
		\State \textbf{Initialize:} 
		\State $r_0=-\mathbb{E}(y(T,\cdot;\alpha; y_{0}; 0))-\Lambda(z_{0})$ (an initial residual)
		\State $p_0 = r_0$ (initial descent direction)\State $k=0$
		
		\While{$||r_k|| > \text{tol}$ and $k < k_{\text{max}}$}
		\State $a_k = \frac{\|r_k\|^2}{\langle \Lambda(p_k), p_k \rangle}$
		\State $z_{k+1} = z_k + a_k p_k$
		\State $r_{k+1} = r_k - a_k \Lambda(p_k)$
		
		\If{$||r_{k+1}|| \leq \text{tol}$}
		\State \textbf{break}
		\EndIf
  
		\State $b_k = \frac{\|r_{k+1}\|^2}{\|r_k\|^2}$
		\State $p_{k+1} = r_{k+1} + b_k p_k$
		\State $k = k + 1$
		\EndWhile
		
		\State \textbf{Output:} Approximate solution $z_k$ of $\tilde{z}_{T}$.
	\end{algorithmic}
\end{algorithm}
However, we immediately see from Algorithm \ref{algo1} that this requires finding the averaged states and adjoint equations at each iteration. To do this, we first simulate a sample $\{\alpha_{1},\cdots, \alpha_{M}\}$ of size $M$ drawn from the distribution of the random variable $\alpha$. Then,  to find the state averaged, we solve the equation \eqref{sys:Schr} for each $\alpha_{k}$, $k=1,\cdots,M$ and, using the classical Monte
Carlo estimator $\mathbb{E}_{M}$, we have 
\begin{eqnarray*}
	\mathbb{E}_{M}(y(t, \cdot;\alpha;y_{0};u))\overset{\text{def}}{=}\frac{1}{M}\sum_{k=1}^{M}y(t, \cdot;\alpha_k;y_{0};u) \approx \mathbb{E}(y(t,\cdot;\alpha;y_{0};u))\quad \text{as} \quad M\longrightarrow \infty.
\end{eqnarray*}
More specifically, for all $t\in [0,T]$ the statistical error in the $L^2$-setting is given by (see \cite{ali2017multilevel, lazar2022control}):
\begin{eqnarray*}
    \|\mathbb{E}_{M}(y(t, \cdot;\alpha;y_{0};u))-\mathbb{E}(y(t,\cdot;\alpha;y_{0};u))\|_{L^{2}(\Omega;L^{2}(G))} \leq \frac{\|y(t, \cdot;\alpha;y_{0};u)\|_{L^{2}(\Omega;L^{2}(G))}}{\sqrt{M}},\;\; \forall M\in \mathbb{N}\setminus \{0\}.
\end{eqnarray*}
\paragraph{}
To solve the equation \eqref{sys:Schr} for $\alpha_{k}$, we use an implicit finite-difference scheme. In this approach, we use the uniform spatial and temporal grid given by $x_j=j\Delta x, j=0,\cdots, Nx$, $t_n=n\Delta t, n=0,\cdots, Nt$ with $\Delta x=\frac{|G|}{Nx}$ and $\Delta t=\frac{T}{Nt}$. Next, we
denote by $y_j^n=y(t_n, x_j)$.
The time derivative is approximated using a backward difference:
\[ y_t \approx \frac{y_{j}^{n+1}-y_j^n}{\Delta t}\]
and the second derivative with respect to $x$ is approximated at time $n+1$ by:
\[ y_{xx} \approx \frac{y_{j+1}^{n+1}-2y_j^{n+1}-y_{j-1}^{n+1}}{(\Delta x)^2}.\]
This leads to a system of linear equations whose representative matrix is tridiagonal.
\subsection{Numerical experiments}
We will solve numerically the averaged null controllability problem of system \eqref{sys:Schr}
and check that the previous CG Algorithm converges satisfactorily in several particular cases.
\subsubsection{Test 1 (Normal distribution)}  
The CG Algorithm has been applied (see Figures \ref{fig: fig1}--\ref{fig: fig4}) with the following data:
\begin{itemize}
	\item $G=(0,1)$, $G_0=(0.25,0.75)$, $T=0.4$.
	\item $y_{0}(x)=\sin(\pi x)$.
	\item $\alpha$ is given by the normal distribution.
\end{itemize}
\par
For our computations, we take $Nx=40$ and $Nt=80$ for the spatial and temporal parameters of the mesh. The initial guess in the algorithm is taken as $z_0=\sin(\pi x)$. We also choose the stopping parameters $k_{\text{max}}=100$ and $\text{tol}=10^{-5}$ for the plots.
\begin{figure}[H]
    \centering
    \includegraphics[width=0.3\linewidth]{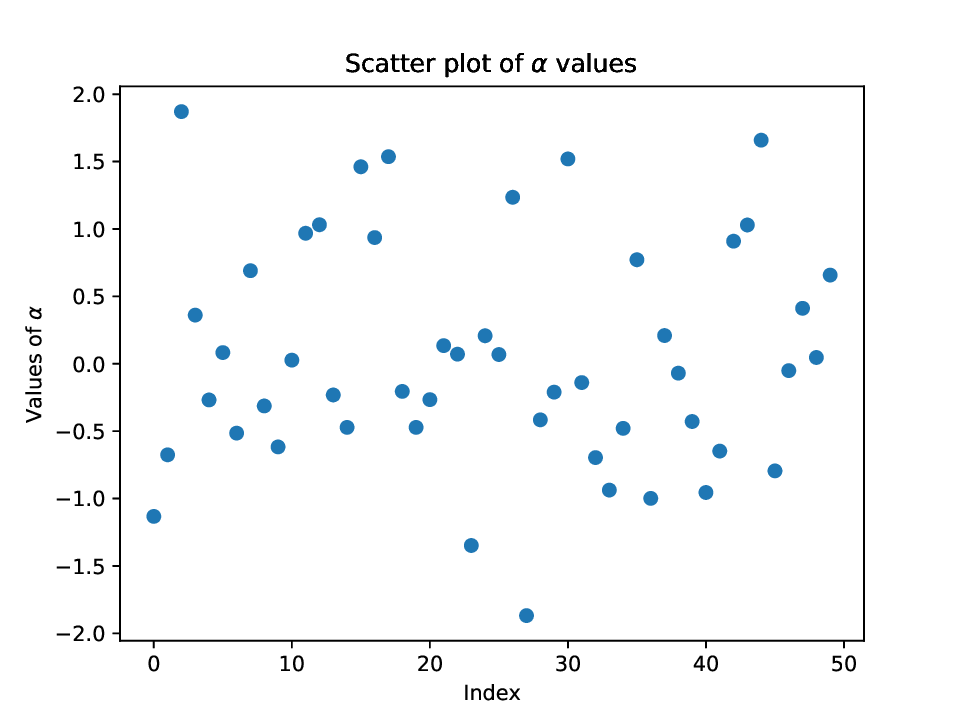}
    \caption{Sample of normal distribution of size $M=50$.}
    \label{fig: fig1}
\end{figure}

\begin{figure}[H]
    \centering
    \includegraphics[width=1.0\linewidth]{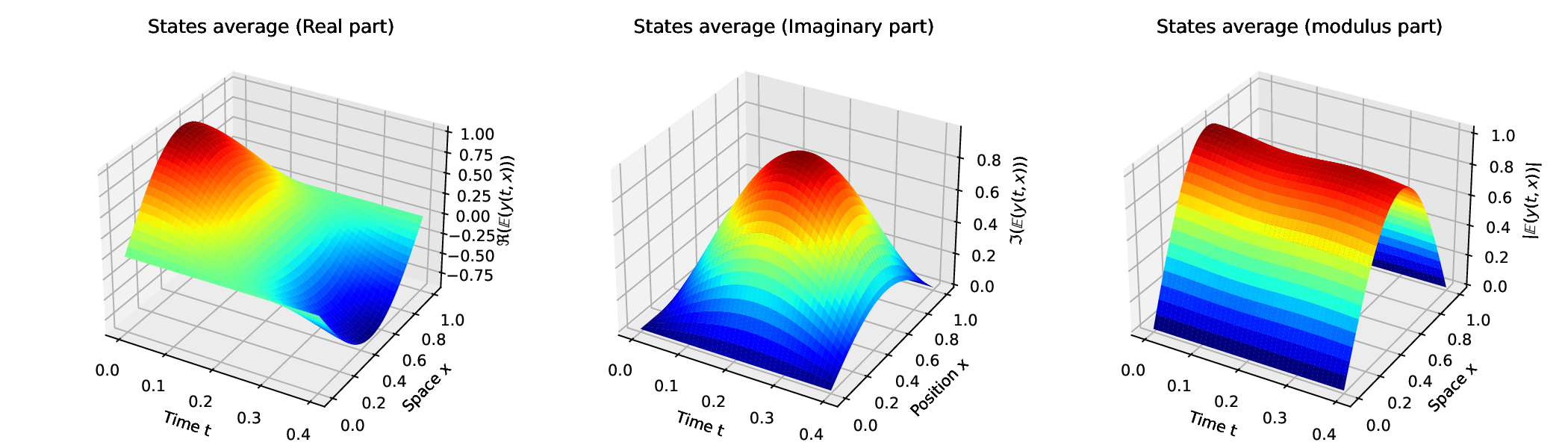}
    \caption{Average of uncontrolled states for normal distribution.}
    \label{fig: fig2}
\end{figure}

\begin{figure}[H]
	\centering
	\includegraphics[width=1.0\linewidth]{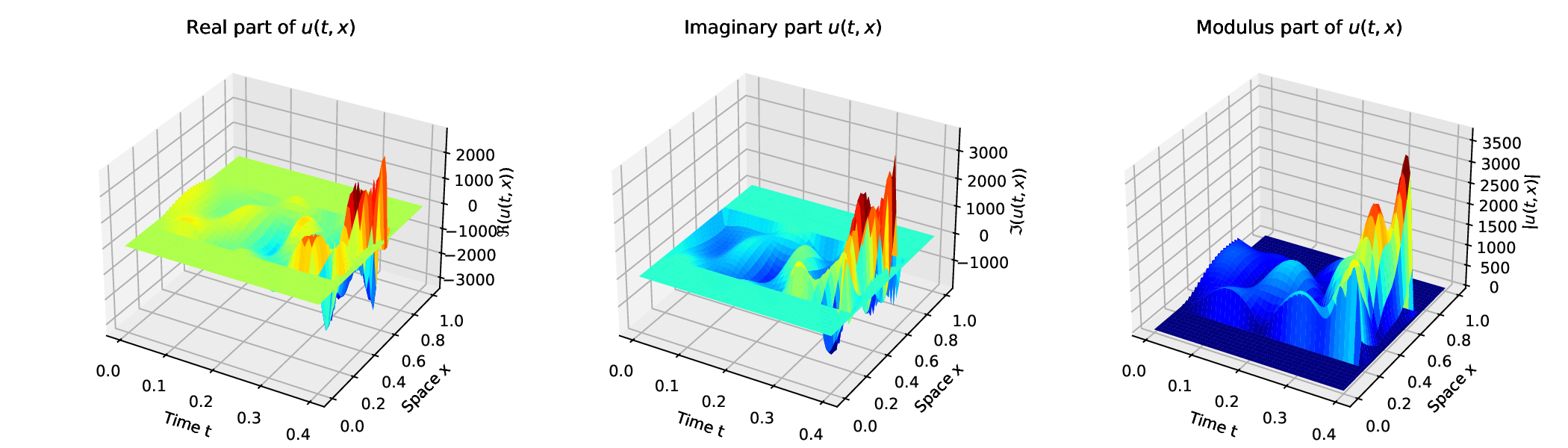}
	\caption{Computed control for normal distribution}
	\label{fig: fig2g}
\end{figure}

\begin{figure}[H]
    \centering
    \includegraphics[width=1.0\linewidth]{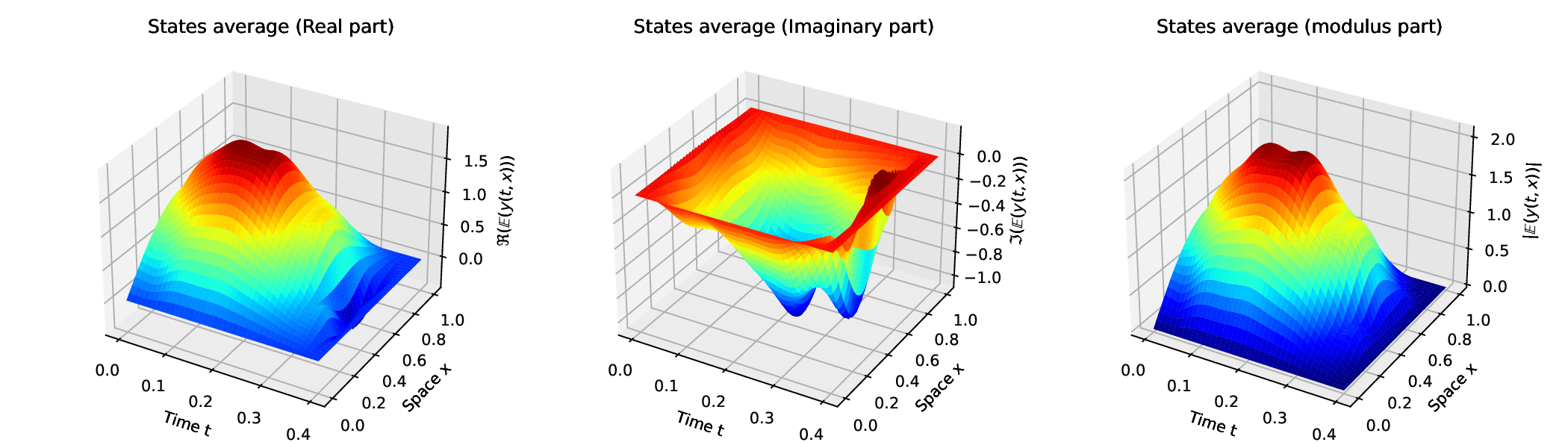}
    \caption{Average of controlled states for normal distribution.}
    \label{fig: fig3}
\end{figure}

\begin{figure}[H]
    \centering
    \includegraphics[width=1.0\linewidth]{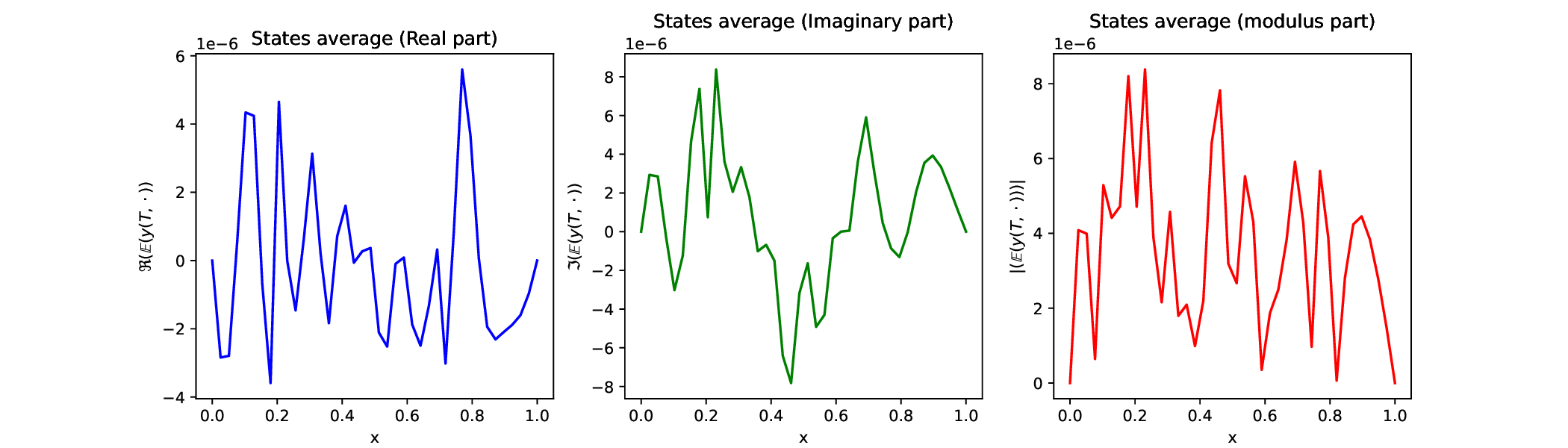}
    \caption{Average of controlled states at time $t=T$ for normal distribution.}
    \label{fig: fig4}
\end{figure}
\subsubsection{Test 2 (Cauchy distribution)} 
In a second experiment (see Figures \ref{fig: fig5}--\ref{fig: fig8}), we kept the data from Test 1, with the exception of the following. 
\begin{itemize}
	\item $\alpha$ is given by the standard Cauchy distribution and $T=0.2$.
\end{itemize}
\begin{figure}[H]
    \centering
    \includegraphics[width=0.3\linewidth]{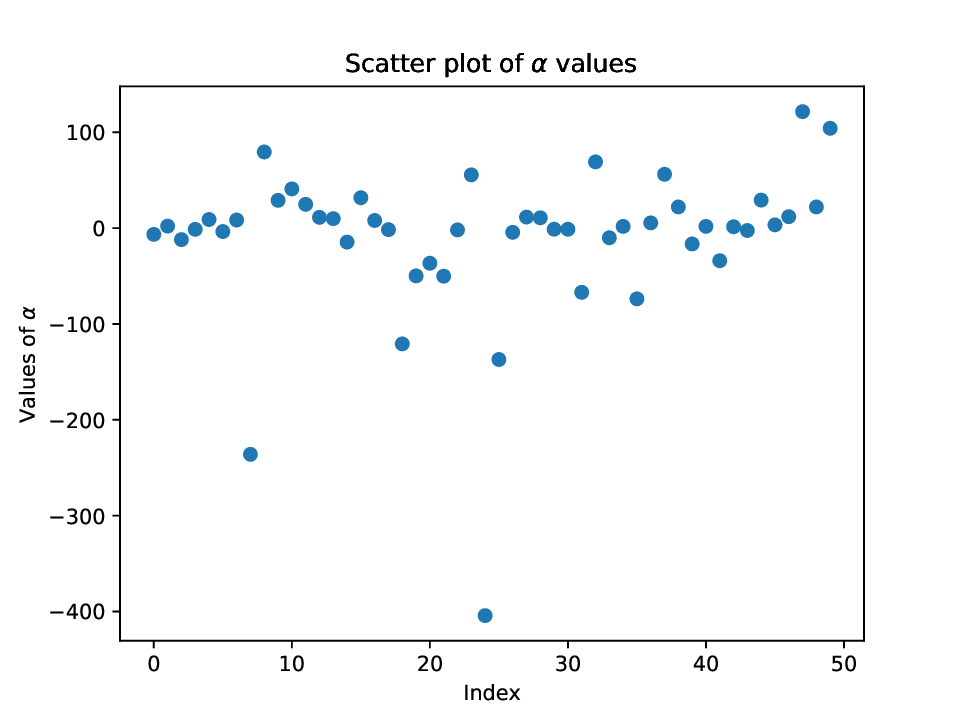}
    \caption{Sample of Cauchy distribution of size $M=50$.}
    \label{fig: fig5}
\end{figure}

\begin{figure}[H]
    \centering
    \includegraphics[width=1.0\linewidth]{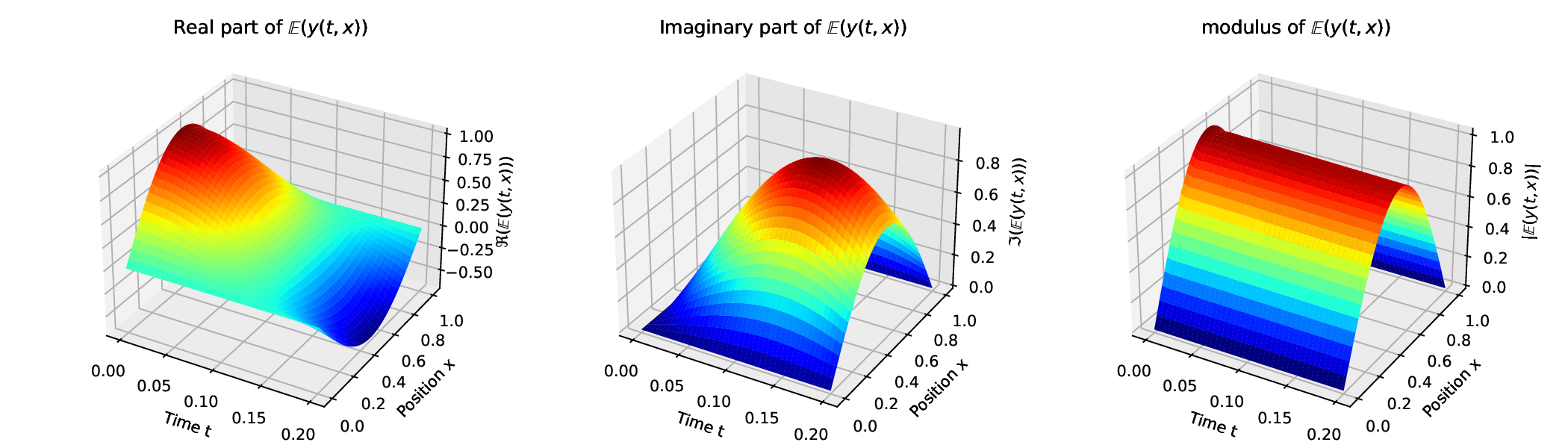}
    \caption{Average of uncontrolled states for Cauchy distribution.}
    \label{fig: fig6}
\end{figure}

\begin{figure}[H]
	\centering
	\includegraphics[width=1.0\linewidth]{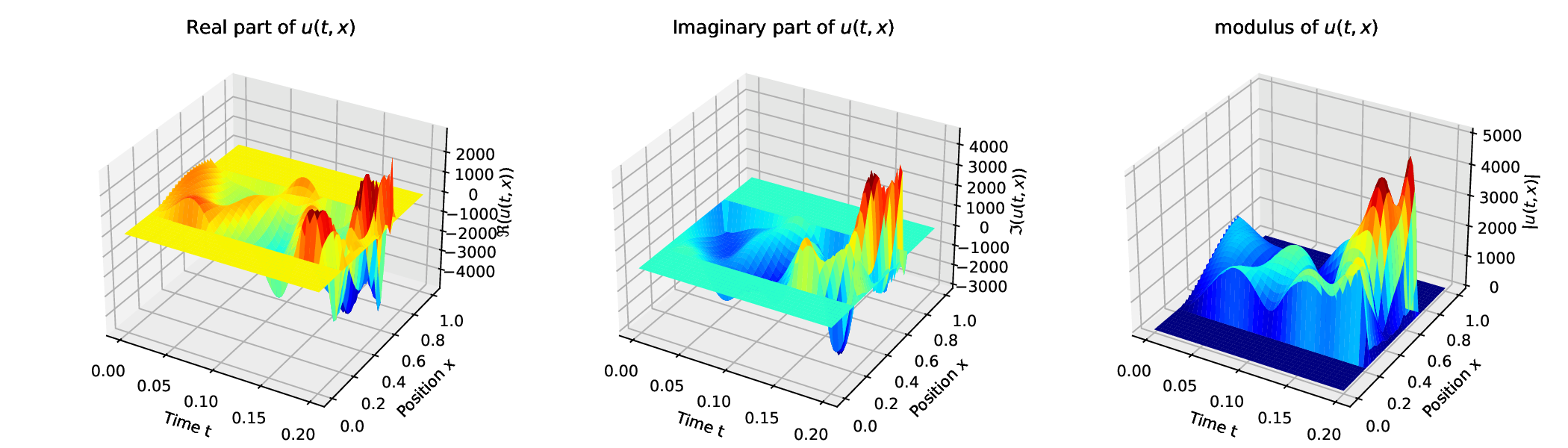}
	\caption{Computed control for Cauchy distribution.}
	\label{fig: fig7g}
\end{figure}

\begin{figure}[H]
    \centering
    \includegraphics[width=1.0\linewidth]{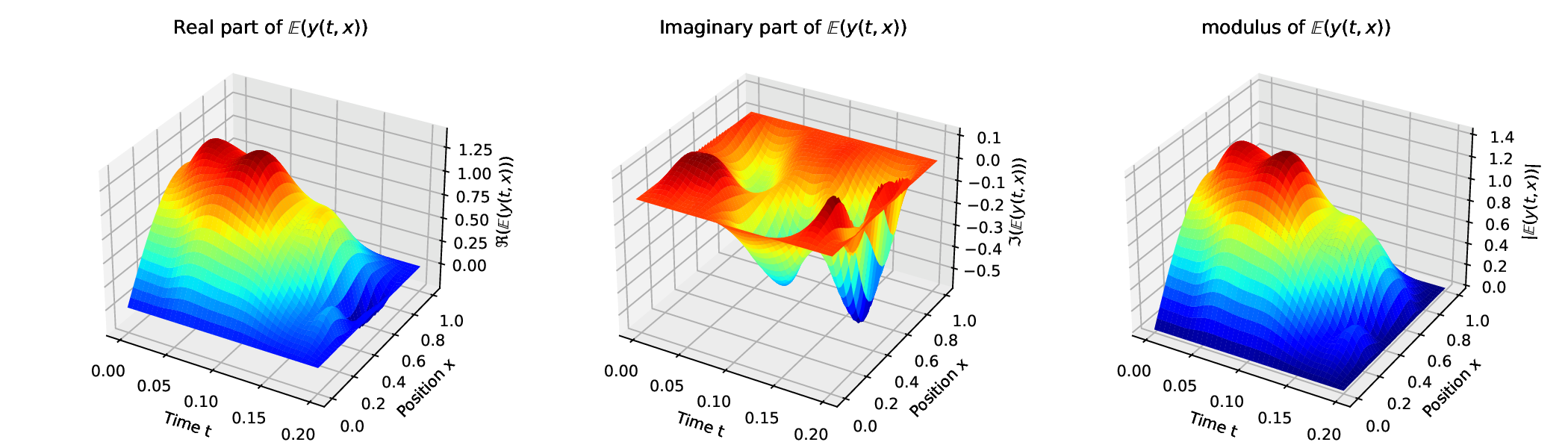}
    \caption{Average of controlled states for Cauchy distribution.}
    \label{fig: fig7}
\end{figure}

\begin{figure}[H]
    \centering
    \includegraphics[width=1.0\linewidth]{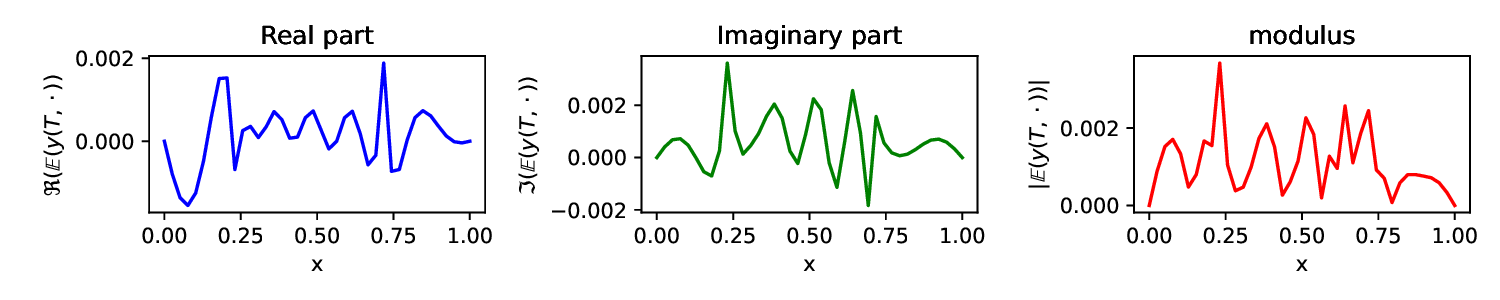}
    \caption{Average of controlled states at time $t=T$ for Cauchy distribution.}
    \label{fig: fig8}
\end{figure}

\par 
We can observe that the controllability error is on the order of $10^{-3}$ for the Cauchy distribution (see Figure \ref{fig: fig8}), while it is reduced to the order of $10^{-6}$ (see Figure \ref{fig: fig4}) for the normal distribution. This difference can be explained by the fact that the Cauchy distribution is highly spread out (see Figure \ref{fig: fig5}) and has heavy tails compared to the normal distribution. In fact, the normal distribution exhibits an exponential decay in the probability of extreme values, which significantly limits the impact of outliers on the error. In contrast, the Cauchy distribution follows a power-law decay, making extreme values more frequent and thereby amplifying the observed error.

\section{Additional control problems}\label{sec:prob}
\paragraph{}
In this Section, we would like to describe the analogue controllability 
problems, and point out some open problems:

\begin{itemize}
\item \textbf{Schrödinger equation with other probability distributions.} Determining the null controllability of the random Schrödinger equation with many probability distributions remain an open problem. This is the case for random variable of stable distributions with $r\leq \frac{1}{2}$. Notably, if $(r,\beta)=(1/2,1)$, we find the Lévy distribution. In spectroscopy, this distribution, with frequency as a dependent variable, is known as a van der Waals profile. For such a distribution the question of null controllability in average remains open and, even if it is a limiting case, is not cover by Hypothesis \HypH. 

\item \textbf{Controls acting on the boundary.} As for controls acting on a measurable boundary subset $\Gamma_0$, 
when $\Gamma_0$ is relatively open, 
similar techniques to the ones in this paper might be used,
but with Theorem 9 in \cite{apraiz2014observability} replacing Lemma \ref{lm:obsellG0}. However, the case in which $\Gamma_0$ is just a set with a strictly positive relative measure requires a much more careful analysis, since one does not have an analogue result to Theorem 5 in \cite{apraiz2014observability}. We note that the boundary-averaged control for random Schrödinger equations, where the diffusivity follows standard random variables and the control region is relatively open, is established in Theorems 5.3 and 5.4 of \cite{lu2016averaged}.

\item  \textbf{Random initial values.} When the random initial datum 
$y_0\in L^1(\Omega; \mathcal{H})$ is independent of the random variable 
$\alpha$, the formula \eqref{eq:Fourdec} remains valid (with 
$y_0$ replaced by 
$\mathbb{E}(y_0)$), since the expectation of the product of independent random variables equals the product of their expectations. As a result, the analogue of Theorem \ref{Main result} still holds. However, when $y_0$ and $\alpha$ are not independent, the characteristic function no longer appears in formula \eqref{eq:Fourdec}, and the problem then requires additional assumptions and a different type of analysis, which goes beyond the scope of this paper.

\item  \textbf{On exact controllability} Even if \eqref{sys:Schr} is not exactly controllable in average in the space $L^2(G)$ with controls acting in $L^2(0,T;L^2(G_0))$, it may be controllable if we consider a larger control space or a smaller target space. For example, as shown in \cite[Theorem 5.3]{lu2016averaged}, when  $\alpha$ is a uniformly distributed random variable, \eqref{sys:Schr} is exactly controllable in average in the space $L^2(G)$ with controls acting in $L^2(0,T;H^{-2}(G_0))$. Also,  as shown in \cite[Theorem 5.3]{lu2016averaged}, when  $\alpha$ is an exponentially distributed random variable, we may reach any final value in $H^2(G)\cap H_0^1(G)$ with initial values in $L^2(G)$ with controls acting in  $L^2(0,T;L^2(G_0))$. 
Finding a characterization on the space in which the equation is exactly controllable in average or obtaining sufficient conditions for the random diffusion remain open problems. 
\item  \textbf{Extension to more general linear Schrödinger operators.} The techniques presented in this paper also apply to any Schrödinger operator of the form
$\partial_t - \alpha \mathrm{i} \mathrm{A}$,
where $\mathrm{A}$ is a self-adjoint elliptic operator with compact resolvent that satisfies the elliptic observability \eqref{est:seqobs}. For instance, the same results hold for the degenerate operator in 
$L^2(0,1)$ given by
\begin{align*}
	\mathrm{A} = \frac{d}{dx}\left(x^a \frac{d}{dx}\right), \quad \text{with } a \in [0,2),
\end{align*}
on a suitable domain depending on whether the degeneracy is weak or strong; see \cite{buffe2024optimal}.
\item \textbf{Non-linear and time dependent Schrödinger equation.} When considering random diffusion, 
dealing with non-linear equations or with time dependent lower order terms
are open problems even for the heat equation. In that setting, it is important to develop new methods that go beyond the use of spectral techniques. 
\item \textbf{Numerics.} A theoretical analysis for designing efficient numerical methods to compute the control of the random Schrödinger equation remains an open problem. Unlike the deterministic case, the Gramian operator is not necessarily positive definite (see Eq. (5.2) and Remark 2.7), which prevents guaranteeing comparable convergence rates. This issue might be addressed by employing regularization techniques or by adapting robust methods developed for finite-dimensional systems in \cite{lazar2022control}.

\end{itemize}

\section{Conclusion}\label{Sec6}
We have investigated the averaged null controllability of random Schrödinger equations, where the diffusivity is a random variable with a characteristic function that decays exponentially. Our results show that such systems are null controllable in average from any measurable subset of the domain with strictly positive measure and within an arbitrarily time, while a simultaneous control is not possible for an absolutely continuous diffusivity. Future research will explore the averaged controllability of the Schrödinger equation with discrete random (which is not finite) diffusivity, as well as extensions to other types of PDEs.

\section*{Acknowledgments}
We thank the anonymous referees for their valuable comments and suggestions.
\footnotesize
\bibliographystyle{abbrv} 
\bibliography{Schravg}

\end{document}